\documentclass[a4paper,12pt]{amsart}  
\usepackage{a4wide,amsmath,amssymb,amsthm}     
\usepackage[utf8]{inputenc}
\usepackage{enumitem}
\usepackage{color}

\theoremstyle{plain}      
 
\newtheorem{thm}{Theorem}[section]

\newtheorem{lemma}[thm]{Lemma}     
\newtheorem{prop}[thm]{Proposition}

\theoremstyle{definition}      
\newtheorem{defn}[thm]{Definition}

\newtheorem{remark}[thm]{Remark}

\DeclareMathAlphabet{\doba}{U}{msb}{m}{n}

\gdef\mN{\doba{N}}

\gdef\mR{\doba{R}}

\gdef\mZ{\doba{Z}}

\def\T{\mathrm{T}}

\def\Ric{{\mathrm{Ric}}}

\def\di{{\rm d}}

\def\vol{{\mathrm{vol}}}

\let\<\langle 
\let\>\rangle

\newcommand{\definedas}{\mathrel{\raise.095ex\hbox{\rm :}\mkern-5.2mu=}}

\title{}

\author{Andrei Moroianu, Mihaela Pilca}

\address{Andrei Moroianu \\ Université Paris-Saclay, CNRS,  Laboratoire de mathématiques d'Orsay, 91405, Orsay, France, 
and Institute of Mathematics “Simion Stoilow” of the Romanian Academy, 21 Calea Grivitei, 010702 Bucharest, Romania}
\email{andrei.moroianu@math.cnrs.fr}

\address{Mihaela Pilca\\Fakult\"at f\"ur Mathematik\\
Universit\"at Regensburg\\Universit\"atsstr. 31 
D-93040 Regensburg, Germany, 
and Institute of Mathematics “Simion Stoilow” of the Romanian Academy, 21 Calea Grivitei, 010702 Bucharest, Romania}
\email{mihaela.pilca@mathematik.uni-regensburg.de}



\begin{document}   
	
	\title{Conformal product structures on compact Einstein manifolds}

	\begin{abstract}  
		In this note we generalize our previous result, stating that if $(M_1,g_1)$ and $(M_2,g_2)$ are compact Riemannian manifolds, then any Einstein metric on the product  $M:=M_1\times M_2$ of the form $g=e^{2f_1}g_1+e^{2f_2}g_2$, with $f_1\in C^\infty(M_2)$ and $f_2\in C^\infty(M_1\times M_2)$, is a warped product metric. Namely, we show that the same conclusion holds if we replace the assumption that the manifold $M$ is globally the product of two compact manifolds by the weaker assumption that $M$ is compact and carries a conformal product structure.
	\end{abstract}
	\maketitle

	\section{Introduction}

A linear connection $D$ on a smooth manifold $M$ is called reducible if its holonomy representation is decomposable, or, equivalently, if the tangent bundle decomposes in a direct sum of two non-trivial $D$-parallel subbundles  $\T M=T_1\oplus T_2$. 
If $D$ is moreover torsion-free and preserves a conformal structure $c$ on $M$ (i.e. if it is a Weyl connection), then the pair $(c,D)$ is called a conformal product structure on $M$ \cite{bm2011}.

A Weyl connection $D$ on $(M,c)$ is determined by the choice of an arbitrary metric $g$ in the conformal class, and by a 1-form $\theta$ on $M$, called the Lee form of $D$ with respect to $g$ (see \eqref{ko} below).
If $\theta$ vanishes, then $D$ is just the Levi-Civita connection of $g$. If $\theta$ is exact, then $D$ is the Levi-Civita of some Riemannian metric conformal to $g$, and if $\theta$ is closed, the same property holds locally. Correspondingly, $D$ is called exact or closed if $\theta$ is exact of closed (and it is easily seen that this does not depend on the choice of metric in the conformal class).

The importance of conformal product structures stems from the Merkulov-Schwachhöfer classification \cite{ms1999} of possible holonomy groups of torsion-free connections. Indeed, when restricted to the special case of Weyl connections, the main result in \cite{ms1999} has the following consequence: if $D$ is a non-closed Weyl connection on a conformal manifold $(M,c)$ of dimension different from $4$, then either the holonomy of $D$ is the whole conformal group, or $D$ is reducible (whence defines a conformal product structure).

Examples of conformal product structures can be constructed on products $M=M_1\times M_2$ of two smooth manifolds. They are defined by a Riemannian metric of the form \begin{equation}\label{cpm}g=e^{2f_1}g_1+e^{2f_2}g_2,
\end{equation}
where $g_i$ are Riemannian metrics on $M_i$ and $f_i$ are smooth functions on $M_1\times M_2$ for $i\in\{1,2\}$, and by a reducible Weyl connection $D$, with Lee form $-\di^{M_1} f_2-\di^{M_2} f_1$ with respect to $g$.
In this case, $(M,[g],D)$ is called a {\em global} conformal product manifold.
Conversely, every manifold with conformal product structure can be written locally in this way (see for instance \cite[Prop. 2.3]{mp2025}).

The problem that we tackle in this note is the following: {\em Which compact conformal product manifolds $(M,c,D)$ carry an Einstein metric $g$ in the conformal class $c$?} In \cite[Thm. 1.1]{mp2024b}, we proved that $g$ is a warped product metric conformal to a Riemannian metric with reducible holonomy, under two extra assumptions:
\begin{enumerate}
\item[$(A)$]  The conformal product structure is global, i.e. $M=M_1\times M_2$, and $g$ is given by \eqref{cpm}.
\item[$(B)$]  The function $f_1$ only depends on $M_2$.
\end{enumerate}

The proof of \cite[Thm. 1.1]{mp2024b} relies heavily on the global expression of $M$ as a product of two compact manifolds. However, this is a very particular situation, since even though the two $D$-parallel distributions $T_1$ and $T_2$ are always integrable on any conformal product, in general their integral leaves are not compact, e.g. on LCP manifolds.

In the present note, we drop completely Assumption $(A)$, and replace $(B)$, which no longer makes sense in the general setting, by:

$(B')$ The restriction of the Lee form of $D$ with respect to $g$ to one of the two distributions $T_1$ or $T_2$ is a closed 1-form on $M$ (in this case the metric $g$ is called special, cf. Definition \ref{defgspecial} below).

Our main result can be stated as follows:

	\begin{thm} \label{mainthm}
		Let $(M,g)$ be a compact Einstein manifold of dimension greater or equal to~$3$ endowed with a conformal product structure $D$, with respect to which the metric $g$ is special. Then $D$ is the Levi-Civita of a Riemannian metric with reducible holonomy conformal to~$g$. Moreover, the conformal factor is constant along one of the $D$-parallel distributions, or equivalently, the pull-back Einstein metric on the universal cover is a warped product metric.
	\end{thm}

Let us now explain the strategy of the proof. Assume that $D$ is a reducible Weyl structure on $(M,g)$, with $D$-parallel orthogonal direct sum decomposition $\T M=T_1\oplus M_2$. The Lee form of $D$ with respect to $g$ can be written as $\theta=\theta_1+\theta_2$, where $\theta_i:=\theta|_{T_i}$ for $i\in\{1,2\}$. The assumption that $g$ is special just means (up to a change of subscripts) that $\theta_2$ is closed. The crucial step is to show that $\theta_1$ is closed too, i.e. $D$ is closed. Once we know this, we are left with the following three cases:
\begin{itemize} 
\item $\theta=0$. Then $D$ is the Levi-Civita connection of $g$, which is thus a reducible Einstein metric. In this case the universal cover of $(M,g)$ is a Riemannian product of Einstein manifolds with the same Einstein constant. 
\item $\theta$ is exact but non-zero. Then $D$ is the Levi-Civita of a Riemannian metric $g'$ with reducible holonomy, which is globally conformal to $g$, but not homothetic to $g$. By \cite[Lemma 5.2 and 5.3]{m2019}, it follows directly that the ranks of the distributions $T_1$ and $T_2$ are at least 2, and that the conformal factor is constant along one of them. In this case the universal cover of $(M,g)$ is globally conformal to a Riemannian product metric, but a complete classification of conformally Einstein product metrics still lacks (see \cite{k1988}, \cite{kr1997}, \cite{kr2009}, \cite{kr2016}).
\item  $\theta$ is closed but non-exact. This corresponds to the case where $(M,g)$ has an LCP structure \cite{brice2024}. It was shown in \cite[Thm. 4.5]{bfm2023} that there are no Einstein metrics on compact LCP manifolds.
\end{itemize}
In order to prove that $\theta_1$ is closed, we distinguish three cases as in the proof of \cite[Lemma 3.2]{mp2024b}: if the ranks of $T_1$ and $T_2$ are at least 2, or if one of the ranks is 1. In the first case, the argument in the proof of Case 1. of \cite[Lemma 3.2]{mp2024b} can be applied since it is of local nature and does not use compactness. However, in the two other cases, \cite[Lemma 3.2]{mp2024b} relies in a crucial way on the global expression of $M$ as a product of two compact manifolds, so new arguments, based on the Bochner formula and maximum principle, are needed in order to treat the general case when the conformal product is not global.

 {\sc Acknowledgments.} This work was partly supported by the PNRR-III-C9-2023-I8 grant CF 149/31.07.2023 {\em Conformal Aspects of Geometry and Dynamics}.

	\section{Preliminaries}

	\subsection{Conformal product structures} Let $(M,g)$ be a Riemannian manifold. We denote  by $\nabla$ the Levi-Civita connection of $g$, and by $\sharp:\T^*M\to \T M$ and $\flat:\T M\to \T^* M$ the musical isomorphisms defined by $g$, which are $\nabla$-parallel and inverse to each other. In order to simplify notation, we will sometimes denote the metric $g$  by $\langle\cdot,\cdot\rangle$ and the associated norm by $\|\cdot\|$ and we will identify vector fields with their dual $1$-forms with respect to $g$, when there is no risk of confusion.
	\begin{defn}\label{weyl} A {\em Weyl connection} on $(M,g)$ is a torsion-free linear connection $D$  satisfying $Dg=-2\theta\otimes g$ for some $1$-form $\theta\in\Omega^1(M)$, called the {\em Lee form} of $D$ with respect to $g$. 
	\end{defn}

	The following conformal analogue of the Koszul formula~\cite{g1995}, shows that the Lee form determines the Weyl connection:
	\begin{equation}\label{ko} D_XY=\nabla_XY+\theta(Y)X+\theta(X)Y-\langle X, Y\rangle\theta^\sharp,\qquad\forall X,Y\in \Gamma(\T M).
	\end{equation}
	
	Hence, Weyl connections on $(M,g)$ are in one-to-one correspondence to $1$-forms on $M$: to each Weyl connection one associates its Lee form with respect to $g$ and to each $1$-form one associates the Weyl connection defined by ~\eqref{ko}.
	
	A Weyl connection on $(M,g)$ is called closed (resp. exact) if its Lee form is closed (resp. exact). If $\theta=\di\varphi$ is exact, then
	$$D(e^{2\varphi}g)=e^{2\varphi}(2\di\varphi\otimes g+Dg)=e^{2\varphi}(2\theta\otimes g-2\theta\otimes g)=0,$$ 
	so $D$ is the Levi-Civita connection of $e^{2\varphi}g$. Thus exact Weyl connections are Levi-Civita connections of metrics in the conformal class of $g$, and closed Weyl connections are locally (but in general not globally) Levi-Civita connections of metrics conformal to $g$. 
	
	\begin{remark}\label{dg} The above computation also shows that the notion of Weyl connection is conformally invariant, in the sense that if $D$ is a Weyl connection on $(M,g)$ with Lee form $\theta$, then for every conformally equivalent metric $g_1:=e^{2\varphi}g$, $D$ is a Weyl connection on $(M,g_1)$ as well, with Lee form $\theta-\di \varphi$.
	\end{remark}
	
	\begin{defn}\label{cp} A {\em conformal product structure} on $(M,g)$ is a Weyl connection $D$ together with a decomposition of the tangent bundle of $M$ as  \mbox{$\T M=T_1\oplus T_2$}, where $T_1$ and $T_2$ are orthogonal $D$-parallel non-trivial distributions. The {\em rank} of a conformal product structure is defined to be the smallest of the ranks of the two distributions $T_1$ and $T_2$. A conformal product structure $D$ on $(M,g)$ is {\em orientable} if the $D$-parallel distributions $T_1$ and $T_2$ are orientable.
	\end{defn}

Any conformal product structure induces by pull-back an orientable conformal product structure on some finite cover of $(M,g)$ with covering group contained in $\mZ/2\mZ\times \mZ/2\mZ$.

 If $(M,g)$ carries a conformal product structure with Weyl connection $D$ and $D$-parallel orthogonal direct sum decomposition \mbox{$\T M=T_1\oplus T_2$}, then for each $0\le k\le n$, the  bundle of $k$-forms of $M$ splits into direct sums 
$$\Lambda^k(M)=\bigoplus_{p+q=k}\pi_1^*(\Lambda^p T_1)\otimes \pi_2^*(\Lambda^q T_2)=:\Lambda^{p,q}(M),$$
where $\pi_i$ denotes the projection from $\T M$ onto $T_i$, for $i\in\{1,2\}$. Since $D$ is torsion-free and the distributions $T_1,T_2$ are $D$-parallel, they are integrable. Therefore, the exterior differential on $M$ maps $C^\infty(\Lambda^{p,q}(M))$ onto  $C^\infty(\Lambda^{p+1,q}(M)\oplus \Lambda^{p,q+1}(M))$. The projections on the two factors of this direct sum are first order differential operators denoted by $\di_1$ and $\di_2$ and they satisfy the relations:
\begin{equation}\label{dd}
	\di=\di_1+\di_2,\qquad \di_1^2=\di_2^2=\di_1\di_2+\di_2\di_1=0.
\end{equation}
We decompose the Lee form as $\theta=\theta_1+\theta_2$, where $\theta_i:=\theta|_{T_i}$, for $i\in\{1,2\}$.

\begin{remark}\label{remdi}
	On a Riemannian manifold $(M,g)$ endowed with a conformal product structure and Lee form $\theta$, the following identities hold (cf. \cite[Lemma 4.6]{bm2011}):
	$$\di_1\theta_1=0, \quad\di_2\theta_2=0.$$
\end{remark}	

If $D$ is a conformal product structure on $(M,g)$ with Lee form $\theta$, every point of $M$ has a neighbourhood $U$ diffeomorphic to a product $U=U_1\times U_2$ such that $T_i|_U=\T U_i$ and the pull-back of $g$ to $U_1\times U_2$ is of the form $e^{2f_1}g_1+e^{2f_2}g_2$, where $g_i$ are Riemannian metrics on $U_i$ and $f_i\in C^\infty(U_1\times U_2)$ for $i\in\{1,2\}$ (see \cite[Prop. 2.3]{mp2025}). If $U$ can be chosen to be equal to $M$, then $(M,g,D)$ is called a {\em global} conformal product.

\subsection{Closed conformal product structures}\label{lcp} If $M$ is compact and the Weyl connection $D$ of a conformal product structure on $(M,g)$ is closed but not exact, then $(M,[g],D)$ is called a locally conformally product (LCP) manifold. In this case, the lift $\widetilde D$ of $D$ to the universal cover $(\widetilde M, \widetilde g)$ of $(M,g)$ is exact (and reducible by assumption), so there exists a Riemannian metric $h\in[\widetilde g]$ whose Levi-Civita connection is equal to $\widetilde D$. This metric is thus reducible, but it is incomplete, so the global de Rham theorem does not apply, even though $\widetilde M$ is simply connected.

However, Kourganoff \cite{k2019} proved that if $(\widetilde M,h)$ is non-flat, then it still has a global de Rham decomposition $(\widetilde M,h)=\mR^q\times (N,g_N)$, where $\mR^q$ is a flat Euclidean space with $q\ge 1$ and $(N,g_N)$ is an irreducible (incomplete) Riemannian manifold. This is a striking result since it not only shows the existence of a de Rham decomposition of $(\widetilde M,h)$, but also that it has exactly two factors, one of which is flat and complete.

The projection to $M$ of the tangent distribution to the $\mR^q$ factor is called the flat distribution of the LCP manifold $(M,[g],D)$.

\subsection{Special metrics on conformal product structures}

We now define some distinguished classes of metrics on conformal product structures, as follows:

\begin{defn}\label{defgspecial}
	Let $(M,g)$ be a Riemannian manifold endowed with a conformal product structure $D$ with orthogonal $D$-parallel decomposition $\T M=T_1\oplus T_2$. We denote as above by $\theta_1$ and $\theta_2$ the two components of the Lee form of $D$ with respect to $g$. The metric $g$ is called {\it special} with respect to this conformal product structure if one of $\theta_1$ or $\theta_2$ is closed and it is called {\it adapted} if one of $\theta_1$ or $\theta_2$ vanishes identically. 
\end{defn}	

On compact LCP manifolds $(M,c,D)$, a notion of {\it adapted} metrics was also introduced by Flamencourt \cite[Def. 3.8]{brice2024}. He calls a metric $g\in c$ adapted if the Lee form of $D$ with respect to $g$ vanishes on the flat distribution. In fact it is easy to see that the Lee form cannot vanish on the complementary (non-flat) distribution (cf. the proof of Proposition \ref{adsp} below). Therefore, on compact LCP manifolds the two definitions of adapted metrics coincide, so Definition~\ref{defgspecial} can be considered as a generalization of Flamencourt's notion of adapted metrics to the setting of conformal product structures.

Note that applying a convolution and smoothing argument, Flamencourt showed that every LCP manifold carries adapted metrics \cite[Prop.~3.6]{brice2024}. Later on, the authors proved in \cite[Thm. 4.4]{mp2024a} that the Gauduchon metric of a compact LCP manifold is adapted.

We now show that the notions of special and adapted metrics coincide on compact LCP manifolds:

\begin{prop}\label{adsp}
	On a compact LCP manifold, a metric $g$ is adapted if and only if it is special. 
\end{prop}	

\begin{proof}
	Let $(M, c, D)$ be a compact LCP manifold with $D$-parallel decomposition of the tangent bundle $\T M=T_1\oplus T_2$, with the convention that $T_1$ is the flat distribution and $T_2$ is the non-flat distribution (since the two $D$-parallel distributions play non-symmetric roles). 
	
	If $g$ is an adapted metric on $(M, c, D)$, then by definition $\theta_1=0$, so in particular $\di\theta_1=0$, meaning that $g$ is special, according to Definition~\ref{defgspecial}. 
	
	Conversely, if $g$ is assumed to be a special metric on $(M, c, D)$, then $\di\theta_1=\di\theta_2=0$ (since one of them has to vanish by definition, and $\di\theta=0$).  
	
	Consider the pull-back structure $(\widetilde g,\widetilde D)$ to the universal cover $\pi:\widetilde M\to M$, with Lee form $\widetilde\theta=\pi^*\theta$. Since $\theta$ is closed, $\widetilde \theta$ is exact, so one can write $\widetilde\theta=\di\varphi$, for some function $\varphi\in C^\infty(\widetilde M)$. By Remark \ref{dg}, $\widetilde D$ is the Levi-Civita connection of the metric $h:=e^{2\varphi}\widetilde g$ on $\widetilde M$. Since $\di\theta_2=0$ we get $0=\di\di_2\varphi=\di_1\di_2\varphi$ on $\widetilde M$, so $\varphi=\varphi_1+\varphi_2$, with $\varphi_1\in C^\infty(\mR^q)$ and $\varphi_2\in C^\infty(N)$ (according to the decomposition $(\widetilde M,h)=\mR^q\times (N,g_N)$ in \S\ref{lcp}). 
	
	Every element $\gamma\in\pi_1(M)$ acts on $\widetilde M$ isometrically with respect to $\widetilde g$ and homothetically with respect to $h$. Moreover, it has the form $\gamma=(\gamma_1,\gamma_2)$, where $\gamma_1$ is a diffeomorphism of $\mR^q$ and $\gamma_2$ is a diffeomorphism of $N$. Pick $\gamma$ which is a strict homothety of $h$ (this exists since $D$ is not exact). Replacing it with $\gamma^{-1}$ if necessary, one can assume that $\gamma$ is a contraction of $h$, thus $\gamma^*\varphi=\varphi-c$ for some positive real number $c$. 
	
	This relation reads $(\gamma_1^*\varphi_1-\varphi_1)+(\gamma_2^*\varphi_2-\varphi_2)=c$, and since the functions in the brackets only depend on $\mR^q$ and $N$ respectively, they are both constant. However, $\gamma_1$ is a contraction on the complete metric space $\mR^q$, so has a fixed point $x_0$. At this point we get $(\gamma_1^*\varphi_1)(x_0)-\varphi_1(x_0)=0$ so $\gamma_1^*\varphi_1=\varphi_1$ on $\mR^q$. For every $x\in\mR^q$ and $k\in\mN$ we thus get $\varphi_1(x)=\varphi_1(\gamma_1^k(x))$. Using that $\gamma_1^k(x)\to x_0$ as $k\to\infty$, we thus get $\varphi_1(x)=\varphi_1(x_0)$, so $\varphi_1$ is constant, whence $\varphi$ is a function on $N$. Therefore $\theta_1=0$, thus showing that the metric $g$ is adapted.
	\end{proof}	

\section{Conformal product structures of rank $1$}
In this section we obtain some preliminary results about conformal product structures of rank $1$, which will be used in the proof of Theorem~\ref{mainthm}. We start with the following:

\begin{lemma}\label{rem}
	If the rank of an orientable conformal product structure~$D$ on an $n$-dimensional Riemannian manifold $(M,g)$ is equal to $1$, then any unit length vector field $\xi$ spanning the $1$-dimensional distribution satisfies the following equality:
	\begin{equation}\label{koxi}
		\nabla_X\xi=-\theta(\xi)X+\langle X, \xi\rangle\theta,\qquad\forall X\in \Gamma(\T M).
	\end{equation}		
Consequently, we obtain
	\begin{equation}\label{codiffxi}
	\nabla_\theta\xi=0,\qquad	\delta\xi=(n-1)\theta(\xi), \qquad \di\xi=\xi\wedge \theta.
	\end{equation}		
\end{lemma}

\begin{proof}
	Assume for instance that $T_1$ has rank $1$. The orientability assumption implies that there is a unit length vector field spanning $T_1$. Let $\xi$  be such a vector field. Since the distribution $T_1$ is $D$-parallel, there exists a $1$-form $\alpha$ on $M$, such that $D_X\xi=\alpha(X)\xi$, for all $X\in \Gamma(\T M)$. Then \eqref{ko} yields:
	\begin{equation}
		\alpha(X)\xi= D_X\xi=\nabla_X \xi+\theta(\xi)X+\theta(X)\xi-\langle X, \xi\rangle\theta^\sharp,\qquad\forall X\in\Gamma(\T M),
	\end{equation}		
	or, equivalently:
	\begin{equation}\label{eqvectxi}
		\nabla_X \xi=(\alpha(X)-\theta(X))\xi-\theta(\xi)X+\langle X, \xi\rangle\theta^\sharp,\qquad\forall X\in\Gamma(\T M).
	\end{equation}		
	Since $\xi$ has unit length, we have $g(\nabla_X\xi, \xi)=0$. Taking the scalar product with $\xi$ in Equality~\eqref{eqvectxi} thus yields $\alpha(X)-\theta(X)=0$, for all vector fields $X$, so $\alpha=\theta$. Hence \eqref{eqvectxi} yields~\eqref{koxi}. Furthermore, the first part of \eqref{codiffxi} follows by taking $X=\theta$ in \eqref{koxi} and the last two are obtained by direct computation, using a local orthonormal basis $\{e_i\}_{i=\overline{1,n}}$:
	\begin{equation*}
		\delta\xi=-\sum_{i=1}^n e_i\lrcorner \nabla_{e_i}\xi=n\theta(\xi)-\theta(\xi)=(n-1)\theta(\xi),
	\end{equation*}		
	\begin{equation*}
		\di\xi=\sum_{i=1}^n e_i\wedge \nabla_{e_i}\xi=\sum_{i=1}^n e_i\wedge(-\theta(\xi)e_i+\langle e_i, \xi\rangle\theta)=\xi\wedge \theta.
	\end{equation*}		
\end{proof}

We now specialize further to the case where the rank 1 conformal product structure is compatible with an Einstein metric.

\begin{lemma}\label{lemmalambda}
	Let $(M,g)$ be an Einstein manifold of dimension $n\geq 3$ carrying an orientable conformal product structure $D$ of rank $1$ with $D$-parallel orthogonal decomposition $\T M=T_1\oplus T_2$. If $\lambda$ denotes the Einstein constant of $(M,g)$, then the following statements hold:
	\begin{enumerate}
	\item[$(i)$] Any unit length vector field $\xi$ spanning the $1$-dimensional distribution satisfies
	\begin{equation}\label{eqlambda}
		\lambda=(n-2)\xi(\theta(\xi))-\delta\theta.
	\end{equation}
\item[$(ii)$]  If $M$ is compact, then $\lambda\geq 0$.
\item[$(iii)$]  If $M$ is compact and $\lambda=0$, then $\theta(\xi)=0$.
\end{enumerate}
\end{lemma}
\begin{proof}
	$(i)$ By Lemma~\ref{rem}, a unit length vector field $\xi$ spanning the $1$-dimensional distribution satisfies \eqref{koxi}. We then compute using a local orthonormal basis $\{e_i\}_{i=1,n}$  parallel at the point where the computation is done and for any vector field $X\in \Gamma(\T M)$ which is also assumed to be parallel at the given point:
	\begin{equation*}
		\begin{split}
			\lambda\langle X, \xi\rangle&=\mathrm{Ric}(X, \xi)=\sum_{i=1}^n \langle R_{e_i, X}\xi, e_i\rangle=\sum_{i=1}^n \left(\langle \nabla_{e_i}\nabla_X\xi, e_i\rangle-\langle\nabla_X\nabla_{e_i}\xi, e_i\rangle\right)\\
			&\overset{\eqref{koxi}}{=}\sum_{i=1}^n \left(\langle \nabla_{e_i}(-\theta(\xi)X+\langle X, \xi\rangle\theta), e_i\rangle-X(\langle\nabla_{e_i}\xi, e_i\rangle)\right)\\
			&=-X(\theta(\xi))+\langle X,\nabla_\theta \xi\rangle-\langle X, \xi\rangle \delta\theta+X(\delta\xi)\\
			&\overset{\eqref{codiffxi}}{=}-X(\theta(\xi))-\langle X, \xi\rangle \delta\theta+(n-1)X(\theta(\xi))\\
			&=(n-2)X(\theta(\xi))-\langle X, \xi\rangle \delta\theta.
		\end{split}
	\end{equation*}		
	Hence, we obtain the following identity
	\begin{equation*}
		\lambda\xi=(n-2)\di(\theta(\xi))-(\delta\theta)\xi,
	\end{equation*}
	which by taking the scalar product with $\xi$, which is a unit vector field, yields \eqref{eqlambda}.
	
	$(ii)$ If $M$ is compact, then integrating \eqref{eqlambda} over $M$, we obtain:
	$$\displaystyle\lambda \mathrm{Vol}(M,g)=(n-2)\int_M \xi(\theta(\xi))\,\vol_g=(n-2)\int_M \theta(\xi)\delta\xi\, \vol_g\overset{\eqref{codiffxi}}{=}(n-1)(n-2)\int_M (\theta(\xi))^2\,\vol_g,$$
	which shows that $\lambda\geq 0$, because the dimension $n$ is assumed to be greater or equal to $3$. 
	
	$(iii)$ If moreover $\lambda=0$, then the above equality implies that $\theta(\xi)=0$.
\end{proof}	

\section{Proof of Theorem~\ref{mainthm}} 

Let $(M,g)$ be a compact Einstein manifold of dimension $n\geq3$, endowed with a conformal product structure $D$ with orthogonal $D$-parallel decomposition $\T M=T_1\oplus T_2$, with respect to which the metric $g$ is special, {\em i.e.} $\theta_1$ or  $\theta_2$ is closed. In order to fix the notation, we assume that $\di\theta_2=0$. Our aim is to show that $\di\theta_1$ vanishes as well, i.e. the Weyl structure $D$ is closed. Up to taking a finite cover of $(M,g)$, we can assume that the conformal product structure is orientable. Note that this does not change the conclusion of Theorem~\ref{mainthm}.
In order to prove that $\di\theta_1=0$ we consider the following three cases.

{\bf Case 1:  $\mathrm{rank}(T_1)\geq 2$ and $\mathrm{rank}(T_2)\geq 2$.} Then the argument in the proof of Case 1. of \cite[Lemma 3.2]{mp2024b} can be applied since it is of local nature, and in particular does not require compactness. Indeed, as mentioned before, every point of a manifold with conformal product structure has a neighbourhood carrying a global conformal product structure, thus satisfying all hypotheses of \cite[Thm 1.1]{mp2024b}, except compactness, which is not needed in the argument in the case where the rank of the conformal product structure is at least 2. With respect of the local form of the metric given by \eqref{cpm}, we have $\theta_1=-\di_1 f_2$. 
By \cite[Lemma 3.2]{mp2024b} we thus obtain that $\di_1\di_2 f_2=0$ in the neighbourhood of every point of $M$, which implies that $\di\theta_1=0$ everywhere on $M$.

{\bf Case 2: $\mathrm{rank}(T_1)=1$.} Let $\xi$ denote a unit vector field spanning the $1$-dimensional oriented distribution $T_1$. The Lee form $\theta$ decomposes then as $\theta=\theta(\xi)\xi+\theta_2$, and $\theta_2(\xi)=0$. Lemma~\ref{lemmalambda} implies that the Einstein constant $\lambda$ of $(M,g)$ is non-negative. We will show that in fact in this case $\lambda$ must vanish. 

Assume for a contradiction that $\lambda>0$. By Myers' Theorem, $\pi_1(M)$ is finite, so the universal cover $\widetilde M$ of $M$ is compact. We consider the conformal product structure $\T\widetilde M=\widetilde T_1\oplus\widetilde T_2$ with Lee form $\tilde\theta$ induced by pull-back on $\widetilde M$, and denote by $\widetilde \xi$ the unit vector field tangent to $\widetilde T_1$ projecting onto $\xi$. The assumption $\di\theta_2=0$ implies that there exists a function $\varphi\in\mathcal{C}^{\infty}(\widetilde M)$ such that $\widetilde\theta_2=\di\varphi$. By \eqref{codiffxi} we then have:
$$\di\widetilde\xi=\widetilde\xi\wedge\widetilde\theta=\widetilde\xi\wedge\widetilde\theta_2=\widetilde\xi\wedge\di\varphi,$$
which implies that $\displaystyle\di (e^\varphi \widetilde\xi)=e^\varphi(\di\varphi\wedge\widetilde\xi+\di\widetilde\xi)=0.$ Hence there exists a function $\psi\in\mathcal{C}^{\infty}(\widetilde M)$, such that 
$e^\varphi \widetilde\xi=\di\psi.$

Since $\widetilde M$ is compact, the function $\psi$ has a global extremum, for which the right-hand side of the above equality vanishes. But this contradicts the fact that the left-hand side is nowhere vanishing, since $\xi$ is a unit vector field. Hence, our assumption is false, which yields $\lambda=0$.

According to Lemma~\ref{lemmalambda}, $\lambda=0$ implies that $\theta(\xi)=0$. Hence $\theta_1$ vanishes, so in particular $\di\theta_1=0$. 

{\bf Case 3: $\mathrm{rank}(T_2)=1$.} Let $\xi$ denote a unit vector field spanning the distribution $T_2$. Then we have $\theta_2=\theta(\xi)\xi$ and thus $\theta_1=\theta-\theta(\xi)\xi$. The assumption $\di\theta_2=0$ reads:
\begin{equation}\label{eqassumpt}
0=\di\theta_2=\di(\theta(\xi))\wedge\xi+\theta(\xi)\di\xi\overset{\eqref{codiffxi}}{=}\di(\theta(\xi))\wedge\xi+\theta(\xi)\xi\wedge\theta=\xi\wedge (\theta(\xi)\theta-\di(\theta(\xi))).
\end{equation}

Let us denote by $f :=\theta(\xi)$ in order to simplify the notation. Then \eqref{eqassumpt} reads
\begin{equation}\label{eqassumpt1}
	\xi\wedge (f\theta-\di f)=0.
\end{equation}

{\bf Claim A.}  The following equality holds:
\begin{equation}\label{thetaxi}
	f\theta\wedge\xi=0.
\end{equation}		
{\it Proof of Claim A.}  We will use the Bochner formula applied to $\xi$:
\begin{equation}\label{bochner}
	\Delta\xi=\nabla^*\nabla\xi+\Ric(\xi).
\end{equation}	
Let $\{e_i\}_{i=\overline{1,n}}$ be a local orthonormal basis parallel at the point where the computation is done. We start by expressing the left-hand side of \eqref{bochner}:
\begin{equation}\label{lapl}
	\begin{split}
	\Delta \xi&=\di\delta\xi+\delta\di\xi\overset{\eqref{codiffxi}}{=}(n-1)\di f-\sum_{i=1}^n e_i\lrcorner(\nabla_{e_i} \xi\wedge\theta+\xi\wedge\nabla_{e_i}\theta)\\
	&=(n-1)\di f+(\delta\xi)\theta+\nabla_\theta\xi-\nabla_\xi\theta-(\delta\theta)\xi\\
	&=(n-1)\di f+(n-1)f\theta-\nabla_\xi\theta-(\delta\theta)\xi,
	\end{split}
\end{equation}	
where for the last equality we used again \eqref{codiffxi}.
We now compute  the connection Laplacian:
\begin{equation}\label{connlapl}
	\nabla^*\nabla \xi=-\sum_{i=1}^{n}\nabla_{e_i}\nabla_{e_i}\xi\overset{\eqref{koxi}}{=}\nabla_{e_i}(f e_i-\langle e_i, \xi\rangle\theta)=\di f+(n-1)f\theta-\nabla_{\xi}\theta.
\end{equation}	
Since $g$ is Einstein with Einstein constant $\lambda$, we have in particular:
\begin{equation}\label{ric}
	\Ric(\xi)=\lambda \xi.
\end{equation}	
Altogether, applying the identities \eqref{lapl}-\eqref{ric}, the Bochner formula \eqref{bochner} reads:
$$(n-2)\di f=(\lambda+\delta\theta)\xi,$$
which, since $n\geq 3$, implies that $\xi\wedge\di f=0$.  Substituting this into \eqref{eqassumpt1}  proves Claim A.\qed

{\bf Claim B.} The following equalities hold:
\begin{equation}\label{claim}
	\lambda=0, \quad \theta(\xi)=0, \quad \delta\theta=0.
\end{equation}	
{\it Proof of Claim B.} 
Let $x_0\in M$ be an extremum point of the function $f=\theta(\xi)$ defined on the compact manifold $M$. We assume by contradiction that $f(x_0)\neq 0$. Then there exists a neighbourhood $U$ of $x_0$, such that $f|_U\neq 0$, which together with  \eqref{thetaxi} implies that $\theta\wedge\xi=0$ on $U$. Hence, on $U$ we have $\theta=f\xi$ and
$$\delta\theta=f\delta\xi-\xi( f)\overset{\eqref{codiffxi}}{=}(n-1)f\theta(\xi)-\xi( f)=(n-1)f^2-\xi(f).$$
 On the other hand, by \eqref{eqlambda} we have $\delta\theta=(n-2)\xi(f)-\lambda$, which replaced in the above equality shows that the following identity holds on $U$:
 $$\lambda+(n-1)f^2=(n-1)\xi(f).$$
In particular at the extremum point $x_0\in U$ we obtain
\begin{equation}\label{lambda}
\lambda+(n-1)f^2(x_0)=0,
\end{equation}
which leads to a contradiction, because $\lambda\geq 0$ by Lemma~\ref{lemmalambda} and $f(x_0)\neq 0$. Hence our assumption was false, which implies that the function $f=\theta(\xi)$ vanishes at all its extremum points and is thus identically zero on $M$. Equality \eqref{eqlambda} then reads $\lambda=-\delta\theta$. Since $\lambda\ge 0$ by Lemma \ref{lemmalambda} (2), integrating over $M$ yields $\lambda=0$, whence $\delta\theta=0$, thus finishing the proof. \qed

Claim B together with \eqref{koxi} yields
 \begin{equation}\label{koxi1}
	\nabla_X\xi=\langle X, \xi\rangle\theta,\qquad\forall X\in \Gamma(\T M).
\end{equation}		
This further implies that $\delta\xi=0$. Applying the exterior derivative to the third equality of~\eqref{codiffxi}, we obtain:
$$0=\di^2\xi=\di(\xi\wedge\theta)=\di\xi\wedge\theta-\xi\wedge\di\theta=\xi\wedge\theta\wedge\theta-\xi\wedge\di\theta=-\xi\wedge\di\theta.$$
Hence, $\di\theta\wedge\xi=0$, showing that there exists a form $\alpha\in\Omega^1(M)$, such that
\begin{equation}\label{dtheta}
\di\theta=\xi\wedge \alpha.
\end{equation}
The equality determines the form $\alpha$ up to its component along $\xi$. In order to make $\alpha$ uniquely defined, we impose that $\alpha(\xi)=0$.
We denote by $T:=\nabla\theta$ and compute for every tangent vector $X$:
\begin{equation*}
\begin{split}
\alpha(X)&=(\xi\wedge\alpha)(\xi, X)\overset{\eqref{dtheta}}{=}\di\theta(\xi, X)=\langle \nabla_\xi \theta, X\rangle - \langle \nabla_X \theta, \xi\rangle\\
&\overset{\theta(\xi)=0}{=}\langle T\xi, X\rangle +\langle \theta, \nabla_X\xi\rangle\overset{\eqref{koxi1}}{=}\langle T\xi, X\rangle +|\theta|^2 \langle X,\xi\rangle=\langle T\xi+|\theta|^2\xi, X\rangle,
\end{split}
\end{equation*}
showing that:
\begin{equation}\label{defalpha}
\alpha=T\xi+|\theta|^2\xi.
\end{equation}
%
%

We now compute the following codifferential:
\begin{equation*}
	\begin{split}
-\delta(\nabla_\theta \theta)&=\sum_{i=1}^n \langle \nabla_{e_i}\nabla_\theta \theta, e_i\rangle=\mathrm{Ric}(\theta, \theta)+\sum_{i=1}^n \left(\langle \nabla_\theta \nabla_{e_i} \theta, e_i\rangle + \langle \nabla_{\nabla_{e_i} \theta } \theta, e_i\rangle\right)\\
&\overset{\lambda=0}{=}\sum_{i=1}^n \left(\theta(\langle \nabla_{e_i} \theta, e_i\rangle) + \di\theta(\nabla_{e_i} \theta,e_i)+\langle \nabla_{e_i} \theta, \nabla_{e_i} \theta\rangle\right)\\
&\overset{\eqref{dtheta}}{=}-\theta(\delta\theta)+\sum_{i=1}^n (\xi\wedge \alpha) (\nabla_{e_i} \theta,e_i)+\sum_{i=1}^n |\nabla_{e_i} \theta|^2\\
&\overset{\delta\theta=0}{=} \langle \xi, \nabla_{\alpha}\theta\rangle-\langle\alpha, \nabla_{\xi} \theta\rangle + |\nabla \theta|^2\\
&=\di\theta(\alpha, \xi)+|\nabla \theta|^2\overset{\eqref{dtheta}}{=}  (\xi\wedge\alpha)(\alpha, \xi)+|\nabla \theta|^2\overset{\alpha(\xi)=0}{=} -|\alpha|^2+|T|^2,
	\end{split}	
\end{equation*}		
whence:
\begin{equation}\label{eqdelta}
	\delta(\nabla_\theta \theta)=|\alpha|^2-|T|^2.
\end{equation}
Let $\{\xi, f_i\}_{i=\overline{1, n-1}}$ be a local orthonormal basis of $\T M$. We have  
$|T|^2=|T(\xi)|^2+\displaystyle\sum_{i=1}^{n-1}|T(f_i)|^2$, and 
$$|T(\xi)|^2\overset{\eqref{defalpha}}{=}|\alpha-|\theta|^2\xi|^2\overset{\alpha(\xi)=0}{=}|\alpha|^2+|\theta|^4.$$
By \eqref{eqdelta} we thus get $$\delta(\nabla_\theta \theta)=-|\theta|^4-\sum_{i=1}^{n-1}|T(f_i)|^2\leq -|\theta|^4.$$ 
Integrating this inequality over the compact manifold $M$ yields that $\theta$ vanishes identically, so in particular $\di\theta_1=0$.

In conclusion, we have shown that the Weyl structure $D$ is closed.  If $D$ were not exact, $g$ would be an Einstein metric on the compact LCP manifold $(M,[g],D)$, which is impossible by \cite[Thm. 4.5]{bfm2023}. Thus $D$ has to be exact, so $g$ is globally conformal to a Riemannian metric with reducible holonomy. The last statement of Theorem~\ref{mainthm} follows directly from \cite[Lemma 5.2 and 5.3]{m2019}.
This concludes the proof.

\end{document}